\documentclass{amsart}
\usepackage{amsmath,amssymb,latexsym} 
\usepackage{mathrsfs}
\usepackage{amsfonts}
\usepackage{amscd}
\usepackage{graphicx}
\usepackage[all]{xy}

\newcommand{\sm}{\smallskip}
\newcommand{\ms}{\medskip}

\newcommand{\ot}{\otimes}
\def\gr*{^{{\rm gr}*}}

\newcommand\pa{\partial}

\def\co{\colon}

\let\goth\mathfrak
\def\gg{\goth g} \newcommand{\g}{\gg}

\def\uce{\mathfrak{uce}}

\def\scF{\mathcal F}
\def\scJ{\mathcal J}

\def\scL{\mathcal L}

\def\scF{\mathcal F}

\newcommand\ga{\gamma}
\newcommand\Ga{\Gamma}
\newcommand\de{\delta} 
\newcommand\eps{\epsilon}

 \newcommand\La{\Lambda}
 \newcommand\vphi{\varphi}
\newcommand\rh{\varepsilon}
\newcommand\si{\sigma}

\newcommand\om{\omega} 

\newcommand{\kalg}{k\mathchar45\mathbf{alg}}
\newcommand{\Ralg}{R\mathchar45\mathbf{alg}}

 \DeclareMathOperator{\Aut}{Aut}

 \DeclareMathOperator{\AD}{AD}

 \DeclareMathOperator{\Cent}{Ctd}
 
 \DeclareMathOperator{\cent}{\Cent}

 \DeclareMathOperator{\Der}{Der}

 \DeclareMathOperator{\End}{End}

 \DeclareMathOperator{\Hom}{Hom}

 \DeclareMathOperator{\Id}{Id}
 
 \DeclareMathOperator{\IDer}{IDer}

  \DeclareMathOperator{\Ker}{Ker}

 \DeclareMathOperator{\Span}{Span}
 \DeclareMathOperator{\Spec}{Spec}

\newtheorem{theorem}{Theorem}[section]
\newtheorem{lemma}[theorem]{Lemma}
\newtheorem{lem}[theorem]{Lemma}
\newtheorem{prop}[theorem]{Proposition}
\newtheorem{cor}[theorem]{Corollary}

\theoremstyle{definition}

\newtheorem{example}[theorem]{Example}

\newtheorem{blank}[theorem]{}
\newtheorem{rem}[theorem]{Remark}
\newtheorem{definition}[theorem]{Definition}

\numberwithin{equation}{section} \allowdisplaybreaks

\begin{document}

\title{\'Etale descent of derivations}
\author{E.~Neher}
\address{Department of Mathematics and Statistics, University of Ottawa,
Ottawa, Ontario K1N 6N5, Canada} \email{neher@uottawa.ca}
\thanks{E. Neher wishes to thank NSERC for partial support through a
Discovery grant}

\author{A. Pianzola}
\address{Department of Mathematics, University of Alberta,
    Edmonton, Alberta T6G 2G1, Canada.
    \newline
 \indent Centro de Altos Estudios en Ciencia Exactas, Avenida de Mayo 866, (1084) Buenos Aires, Argentina.}
\thanks{A. Pianzola wishes to thank NSERC and CONICET for their
continuous support}\email{a.pianzola@gmail.com}


\begin{abstract} We study \'etale descent of derivations of algebras with
values in a module. The algebras under consideration are twisted forms of
algebras over rings, and apply to all classes of algebras, notably
associative and Lie algebras, such as the multiloop algebras that appear in the
construction of extended affine Lie algebras. The main result is Theorem~\ref{main}.
\end{abstract}

\maketitle

\section*{Introduction} Let $\gg$ be a simple finite-dimensional
Lie algebra over an algebraically closed field $k$ of characteristic $0$. The celebrated affine Kac-Moody Lie algebras are of the form
$$ \mathcal {E} = \scL \oplus kc \oplus kd$$
where $\scL$ is a (twisted) loop algebra of the form $ L(\gg, \pi)$ for some
diagram automorphism $\pi$ of $\gg.$ The element $c$ is central and $d $ is a
degree derivation for a natural grading of $\scL.$

It is thus natural to study the derivations of loop, or more generally
multiloop algebras. This is a problem with a long history going back to
\cite{bemo,Block}. In some sense there is a prescient aspect to \cite{bemo},
which seems to sense the existence of  Lie algebras that would have multiloop
algebras play the same role than the loop algebras play in the affine case.
These algebras would come into being a decade and a half later in the form of
extended affine Lie algebras (EALAs) and Lie tori (see
\cite{AABGP,n:eala-summer,n:eala}).

The first author has shown that any EALA is built from a Lie torus at the
``bottom'' in a way reminiscent of the affine construction.  Loosely speaking
an EALA is always of the form $$\mathcal{E} = \scL \oplus \mathcal{C} \oplus
\mathcal{D}$$ where $\scL$ is a Lie torus, $\mathcal{C}$ is central in $\scL
\oplus \mathcal{C}$  and $\mathcal {D}$ is a space of derivations of the
bottom Lie torus $\scL$. It is known, save for perfectly understood
exceptions in absolute type $A$, that Lie tori are always multiloop algebras
\cite{ABFP2} (but not conversely, except for nullity $1$ as shown in
\cite{P2}). We begin to see the central importance that the understanding of
the Lie algebra of $k$-linear derivations of multiloop algebras has for EALA
theory. They are also important for structure of universal central extensions
\cite{N1}

Exploiting  the fact that a centreless Lie torus $\scL$ which is finitely
generated over its centroid $R$ is \'etale (even Galois) locally isomorphic
 to the $R$--algebra $\gg \otimes_k R$, there is a very transparent way of understanding the
derivations of $\scL.$ The main idea (see \cite{P} for details) is
disarmingly simple and we outline it here since it will serve as the blue
print for our work: Let $S/R$ be an \'etale covering that trivializes $\scL.$
\sm

(a) As shown in \cite{bemo}, upstairs, namely for $\gg \otimes_k S,$ the picture is perfectly understood:
Besides the inner derivations we have $\Der_k(S)$ acting naturally as derivations of $\gg \otimes_k S.$
\sm

(b)  One can recover $S$ from $\gg \otimes_k S$ as its centroid. Thus (a)
says that the derivations of the centroid of $\gg \otimes_k S$ extend
naturally to $\gg \otimes_k S.$ More precisely, $\Der_k( \gg \ot_k S) \simeq
\IDer(\gg \ot_k S) \rtimes \Der_k(S)$ with $\IDer(\gg \ot_k S) \simeq
\Der_k(\gg) \ot_k S$. 
     \sm

(c) By descent considerations (\cite[Lemma~4.6]{GP2}) the  centroid of $\scL$ is $R$. Every derivation of $\scL$ naturally induces a derivation of $R$, but it is not
obvious that the derivations of $R$ can be lifted to $\scL$ (as it does in the
trivial case (a) above). \sm

(d) Because $S/R$ is \'etale, every $k$-linear derivation of $R$ extends
uniquely to a $k$-linear derivation of $S$, hence to a derivation of $\gg
\otimes_k S$ by (a). \sm

(e) The derivation of $\gg \otimes_k S$ defined in (d) descends to $\scL$.
\sm

The resulting picture is thus completely analogous to the one of the trivial case:
$$\Der_k(\scL) \simeq \IDer_k(\scL) \rtimes \Der_k(R).$$
A close inspection shows that more important than the structure of
$\Der_k(\scL)$ to the theory of EALAs of central extensions is the structure
of $\Der_k(\scL, N)$ where $N$ is the graded dual of $\scL.$ Rather than
studying this particular case we set out to see if the descent formalism will
shed information about the structure of $\Der_k(\scL, N)$ for an {\it
arbitrary} $\scL$--module $N$ and $S/R$--form $\scL$. The answer is a
resounding yes. There are, however, subtle technical and philosophical
difficulties to overcome before one can even state a structure result. This
is already quite evident in (b) and (c) above. While a derivation of $\scL$
induces a derivation of its centroid, what does a derivation in $\Der_k(\scL,
N)$ induce, and on what? As we shall see, proceeding in a natural way will
guide us -- as it always seems to do -- towards the correct answer:
Theorem~\ref{main}.

Although in this introduction we have emphasized Lie algebras, the main
results of this paper will be established for arbitrary algebras. This
substantially broadens the applications of our work. For example, in \ref
{asso-rsult} we use our main theorem to derive a new characterization of the
first Hochschild cohomology group of separable associative algebras. Our approach is new even when specialized to the previously known cases.\medskip

Throughout, $k$ will be a commutative associative unital ring. We denote by
$\kalg$ the category of commutative associative unital $k$-algebras with
unital algebra homomorphisms as morphisms. If $S\in \kalg$ and $N$ is a
$k$--module, we put $N_S = N \ot_k S$. The term {\em $k$-algebra\/} will mean
a $k$-module $A$ together with a $k$-bilinear map $A \times A \to A$, $(a_1,
a_2) \mapsto a_1 a_2$. In particular, we do not require any further
identities  (even though our interest is mostly with associative or Lie
algebras). Also note that the left and right $k$-module structure of $A$ are assumed to coincide: $ca=ac$ for all $c\in k$ and $a\in A$.
\sm

\noindent \textit{Acknowledgments.} The authors would like to thank the referees for their valuable comments.

\section{Dimodules} \label{sec:dimod}

We want to define the concept of a derivation of an {\it arbitrary} $k$-algebra $A$
with coefficients in a $k$--module $M.$  Our guiding principle is the
Leibnitz rule. To make sense of it in the most general way we introduce the
concept of dimodule. Proceeding in this way, we can apply our results to
various classes of algebras (associative, Lie, Jordan, ...).

Let $A$ be a $k$--algebra. An {\it $(A,k)$--dimodule} is a $k$--module $M$
together with a pair of  $k$--bilinear maps  $A \times M \to M, (a,m) \mapsto
a \cdot m$ and $M \times A \to M, (m,a) \mapsto m \cdot a$,  which we call
{\it left} and {\it right} action of $A$ on $M.$\footnote{We intentionally
put no compatibility assumptions between the two actions. The authors are
aware that this rather general concept does not agree with the usual notion
of a bimodule if $A$ is, for example, an associative algebra or a Lie
algebra. We hope that the use of different terminology will avoid any
possible confusion  and that the following results will convince the reader
of the usefulness of this new more general concept.}

\begin{rem}\label{base-change} The notion of an $(A,k)$-dimodule is stable
under base change. Indeed, if $K\in \kalg$ and $M$ is a $(A,k)$-dimodule,
then $M_K$ is naturally a $(A_K,K)$--dimodule by defining
\begin{equation} \label{di-mod-act}
 (a\ot s_1) \cdot (m \ot s_2) = (a \cdot m) \ot s_1s_2 \quad \hbox{and}\quad
 (m \ot s_1) \cdot (a \ot s_2) = (m \cdot a) \ot s_1s_2 \end{equation}
for all $a\in A$, $m\in M$ and $s_1, s_2 \in S$.
\end{rem}

\begin{example}\label{trivial} The algebra $A$ is in a natural way (via left and
right multiplication) an $(A,k)$--dimodule. Similarly, $A^* = \Hom_k(A,k)$ is
canonically an $(A,k)$--dimodule by defining $a \cdot \vphi$ and $\vphi \cdot
a$ for $a\in A$, $\vphi \in A^*$ as follows: $(a\cdot \vphi)(a') =
\vphi(a'a)$ and $(\vphi \cdot a)(a') = \vphi(aa')$ for $a'\in A$.

Suppose that $A$ is in fact an $R$--algebra for some $R\in \kalg$. Then $A^*$
is naturally an  $R$--module via $(r \cdot \vphi )(a) = \vphi(ra)$ for $\vphi
\in A^*$, $a\in A$ and $r\in R$. The $A$--action on $A^*$ is compatible with
the $R$--action in the following sense:
\begin{equation} \label{exa:lie-alg1}
   r \cdot (a \cdot \vphi) = (r a) \cdot \vphi = a \cdot (r \cdot\vphi) \quad
   \hbox{and} \quad
       r \cdot (\vphi \cdot a ) =  (r \cdot\vphi) \cdot a = \vphi  \cdot (r  a).
\end{equation}
Thus $A^*$ is an $(A,R)$--dimodule.
\end{example}

Given two  $(A,k)$--dimodules $M$ and $N$, an $(A,k)$--dimodule morphism $f:
M \to N$ is a $k$--linear map satisfying $f(a\cdot m) = a \cdot f(m)$ and
$f(m \cdot a) = f(m) \cdot a$ for $a\in A$ and $m\in M$. We thus have an
obvious category of $(A,k)$--dimodules for any given algebra $A$.  We leave
it to the reader to define in the  $(A,k)$--dimodule setting the concepts of
submodule, kernel, quotient... \sm

For the remainder of this section $A$ will denote a $k$--algebra and $M$ an
$(A,k)$--dimodule.

\subsection {Derivations}
A {\em derivation\/} of $A$ with values in $M$ is a $k$--linear map $d: A \to
M$ satisfying $d(a_1 a_2) = d(a_1) \cdot a_2 + a_1 \cdot d(a_2)$ for $a_i \in
A$. We denote by $\Der_k(A,M)$ the $k$--module of derivations of $A$ with
values in $M$. Note that $\Der_k(A,A) =: \Der_k(A)$ has a natural $k$--Lie
algebra structure: The commutator of two derivations of $A$ is a derivation
of $A$.

\begin{prop}\label{proppp} Let $A$ and $M$ be as above and let $K\in \kalg$. Then.

{\rm (a)} The canonical map $\om : \Hom_k (A,M) \ot_k K \to \Hom_K
(A_K,M_K)$, given by $f\ot x \mapsto f\ot' x$ where $(f\ot' x)(a\ot y) =
f(a)\ot xy$ for $a\in A$ and $x,y\in K$, maps  $\Der_k (A,M) \ot_k K$ to
$\Der_K(A_K,M_K)$.\sm

{\rm (b)} If $K$ is a flat $k$-module and $A$ is finitely presented as a
$k$-module, then $\om$ is an isomorphism for the $K$-modules considered in
{\rm (a)}, in particular \begin{equation}
 \label{der-extension1}
\Der_k(A,M)_K  \simeq \Der_K(A_K,M_K).
\end{equation}  \end{prop}
\sm

{\bf Note}: To simplify notation we will often write $f \ot x$ for $f\ot' x$.
This is certainly justified in the setting of (b).

\begin{proof} (a) is immediate.
(b) Define $
   \de_A : \Hom_k(A,M) \to  \Hom_k(A\ot_k A, M)
$ as the unique $k$--linear map satisfying
\[
   \de_A(f) \,(a_1 \ot a_2) =  f(a_1a_2) - f(a_1) \cdot a_2 -a_1 \cdot
   f(a_2)
\] and observe $\Der_k(A,M) = \Ker \de_A$. We have the commutative
diagram
$$ \xymatrix{
   0 \ar[r] & \Der_k(A,M)_K \ar[r] \ar[d]^{\om |_{\Der}}
       &\Hom_k(A,M)_K \ar[r]^-{\de_A \ot \Id} \ar[d]^\om
      & \Hom_k(A \ot_k A, M) \ot_k K \ar[d]^\varepsilon\cr
   0 \ar[r] & \Der_K(A_K,M_K) \ar[r]
        & \Hom_K(A_K,M_K) \ar[r]^-{\de_{A\ot K}} \ar[r]
    & \Hom_K(A_K \ot_K A_K, M_K) }
$$
where the top row is exact because $K$ is flat and the bottom row is exact by
definition of $\Der_K(A_K,M_K)$. Bijectivity of $\om|_{\Der}$ now easily
follows since $\om$ is an isomorphism and $\rh$ is injective \cite[I, \S2.10,
Prop.~11]{bou:ACa}.   \end{proof}

\begin{example}[Lie algebras]\label{exa:lie-alg} If $L$ is a Lie algebra, an $L$--module
$M$ (in the usual sense) has a $(L,k)$--dimodule structure (that we call {\it
canonical}) as follows. The left action is the given module action. The right
action is defined by  $m \cdot l = - l \cdot m$ for $l\in L$, $m\in M$.  We
leave it to the reader to check that the definition of $\Der_k(L, M)$
coincides with the usual definition of derivations of $L$ with values in $M.$
We will also consider the subdimodule of $\Der_k(L,M)$ consisting of {\em
inner derivations} of $M$  defined as usual by
 \[ \IDer_k(L,M) = \{ \pa_m : m \in M \}, \quad
  \pa_m (l) = l \cdot m .  \]
\end{example}

\begin{cor}\label{cor-innder-Lie} Let $L$
be a Lie $k$--algebra and $M$  an $L$--module.  \sm

{\rm (a)} The canonical map
\[
    \IDer_k(L,M) \ot_k K \to \IDer_K(L_K, M_K), \quad
          \pa_m \ot s \mapsto \pa_{m \ot s} \]
is a well-defined epimorphisms of $k$-modules. \sm

{\rm (b)} The map in {\rm (a)} is injective, whence an isomorphism, if $K$ is
a flat $k$-module and $L$ is finitely generated as a $k$-module.
 \sm

{\rm (c)} Assume that $K$ is a faithfully flat $k$-module and that $L$ is
finitely presented as $k$-module. Then
$$
    \IDer_k(L,M) = \Der_k(L,M) \quad \iff \quad
           \IDer_k(L_K,M_K) = \Der_K(L_K,M_K).
$$
\end{cor}

\begin{proof} (a) is immediate. Under the hypothesis of (b), the map
$\om$ of Proposition~\ref{proppp} is injective. (c) follows from (b) and the
proposition.   \end{proof}

\begin{example} (Rings with twisted derivations).\footnote{We thank Kirill Zainoulline for pointing out this example.} \label{ex:KP} Even in the setting of associative algebras, derivations into dimodules rather than bimodules arise naturally in invariant theory. Recall \cite[\S1]{kp:gen-inv} that a {\em ring with twisted derivations} is a quadruple $(A, I, (\vphi_i)_{i\in I}, (d_i)_{i\in I})$ consisting of a $k$-algebra $A$, an index set $I$, a family $(\vphi_i)_{i\in I}$ of automorphisms $\vphi_i \in \Aut_k(A)$ and a family $(d_i)_{i\in I}$ of endomorphisms $d_i \in \End_k(A)$ satisfying
\[ d_i(a_1 \, a_2) = d_i(a_1) \, a_2 + \vphi_i(a_1) \, d_i(a_2) \]
for all $a_1, a_2 \in A$. Given $(A, I, (\vphi_i)_{i \in I})$, let $M=\prod_{i\in I} M_i$ where $M_i$ is the $A$-dimodule with $M_i = A$ as $k$-module and $A$-actions given by $a \cdot m_i = \vphi_i(a) m_i$ and $m_i \cdot a = m_i a$ (in both equations we use the multiplication of $A$ on the right hand side). The canonical isomorphism $\Hom_k(A, M) \simeq \prod_{i\in I} \Hom_k(A, M_i)$ induces a bijection between $\Der_k(A,M)$ and the set of all rings of twisted derivations of the form $(A, I, (\vphi_i)_{i\in I}, (d_i)_{i\in I})$.
\end{example}
\sm

\subsection{Centroids} We remind the reader that throughout this section $A$ denotes a $k$--algebra
and $M$ an $(A,k)$-dimodule.

\begin{definition}The {\em centroid
of $A$ with values in $M$\/} is defined as
\[
  \Cent_k(A,M) = \{ \chi \in \Hom_k(A,M): \chi(a_1 a_2) =
       \chi(a_1) \cdot a_2 = a_1 \cdot \chi(a_2) \hbox{ for all }
         a_1, a_2 \in A\}
\]
Observe that $\Cent_k(A,A)$  is the standard centroid of the $k$--algebra
$A$, which justifies our terminology and notation. For any $k$--algebra $A$
there is always a canonical ring homomorphism
\begin{equation} \label{def:cent-map}
 \chi \co k \to \cent_k(A)
\end{equation}
which sends $s\in k$ to $\chi_s$ defined by $\chi_s(a) = sa$ for $a \in A$.
We call $A$ {\em central\/} if $\chi$ is an isomorphism.
\end{definition}

The next results collect some basic but important results about
centroids. The mostly elementary proofs will be left to the reader.

\begin{lem}\label{der-R-mod}
Let $R \in \kalg$, and suppose that $M$ has an $R$--module structure which is compatible with the $(A,k)$--dimodule structure in the following sense,
\begin{equation} \label{der-R-mod1}
    r(m \cdot a) = (rm) \cdot a \quad \hbox{and} \quad
    r(a \cdot m) = a \cdot (rm)
\end{equation}
for $r\in R$, $a\in A$ and $m\in M$. \sm

Assume now that $A$ is also an $R$--algebra. Then the following hold. \sm

{\rm (a)} $\Hom_k(A,M)$ is an $R$--bimodule via \begin{equation}\label{referee1}
   (r \cdot f)(a) = r \big(f(a) \big) \quad \hbox{and} \quad (f \cdot r)(a)
       =   f(r a)
\end{equation}
for $r\in R$, $f\in \Hom_k(A,M)$ and $a\in A$. \sm

{\rm (b)} $\cent_k(A,M)$ is a subbimodule of\/ $\Hom_k(A,M)$.
 \sm

{\rm (c)} $\Der_k(A,M)$ is a submodule with respect to the left $R$--module
structure of $\Hom_k(A,M)$.
\end{lem}

Recall that an algebra $B$ over some $R\in \kalg$ is  {\em perfect\/} if it
equals its derived algebra $B'$, defined to consist of the sums of products
$b_1 b_2$ with $b_i \in B$. It is immediate that the computation of the
derived algebra commutes with base ring extensions: $(B')_S = (B_S)'$. In
particular, if $B$ is perfect then so is $B_S$ for any $S\in \Ralg$. The
converse holds in the following situation.

\begin{lemma} \label{perf-descent} Let $B$ be an algebra over $R\in \kalg$ and let $S\in \Ralg$ be
faithfully flat. Then $B$ is perfect if and only if $B_S$ is.
\end{lemma}

\begin{proof} If $B_S$ is perfect, then $0 = B_S/(B_S)' \simeq (B/B')\ot_R S$
by flatness, whence $B/B' = 0$ by faithful flatness.
  \end{proof}

\begin{lemma}\label{cent-perf} Let $A$ be a perfect $R$--algebra
for some $R\in \kalg$,  and suppose that $M$ is also an $R$--module such that
the $R$--module  and the $(A,k)$--dimodule structures are related by
\begin{equation}\label{cent-perf1}
     r (a \cdot m) = (ra) \cdot m
 \end{equation}
for $r\in R$, $a\in A$ and $m\in M$. Then every $k$--linear centroidal
transformation is already $R$--linear:
\[ \cent_k ( A, M) = \cent_R(A, M).\]
Furthermore, if also \eqref{der-R-mod1} holds, the two $R$--module structures
of\/ $\cent_k(A,M)$\/ defined in {\rm \eqref{referee1}} coincide.
\end{lemma}

We note that the conditions \eqref{der-R-mod1} and \eqref{cent-perf1} are
always fulfilled if $M$ is an $(A,R)$--dimodule, where $M$ is viewed as an
$(A,k)$--dimodule by restriction of scalars.

\begin{lemma} \label{cent-ext} Suppose that $R\in \kalg$, $B$ is an $R$--algebra
whose underlying $R$-module is finitely presented, $N$ is a $(B,R)$--dimodule
and $S\in \Ralg$ is a flat extension. Then the canonical map
\[  \cent_R(B,N) \ot_R S \to \cent_S(B_S, N_S)
    \]
is an isomorphism of $S$-modules, where $N_S$ is the $(B_S, S)$--dimodule
obtained from the $(B,R)$--dimodule $N$ by the base change $S/R$, see
\eqref{di-mod-act}. \end{lemma}

\begin{proof} The proof is similar to that of \cite[Lemma~3.1]{P}, which deals
with the case $N=B$.   \end{proof}

\begin{cor}\label{forms-are-cent} Let $A$ be a central $k$-algebra
which is finitely presented as $k$-module, let $R\in \kalg $ be a flat
extension, $S/R$ a faithfully flat extension and let $B$ be an $S/R$-form of
$A$, i.e., $B$ is an $R$-algebra such that $B \otimes_R S$ is
isomorphic as an $S$--algebra to $(A \otimes_k R) \ot_R S \simeq A \ot_k S$.
Then $B$ is a finitely presented $R$--algebra, and the canonical map
$\chi$ of \eqref{def:cent-map} is an isomorphism. In particular, $B$ is a
central $R$--algebra.
\end{cor}

\begin{proof} Since $S/k$ is flat by transitivity of flatness, Lemma~\ref{cent-ext} shows
$\cent_S(A_S) \simeq \cent_k(A) \ot_k S \simeq S$. Next we observe that $B$
is a finitely presented $R$--algebra since finite presentation is preserved
by faithfully flat descent \cite[I, \S3.6, Prop.~11]{bou:ACa}. Hence, a
second application of Lemma~\ref{cent-ext} yields $\cent_R(B) \ot_R S \simeq
\cent_S(B_S) \simeq \cent_S(A_S) \simeq S$. Thus $\chi$ is an isomorphism by
\cite[I, \S3.1, Prop.~2]{bou:ACa}.
  \end{proof}

\section{Derivations of twisted forms with values in a dimodule}
\label{sec:deri}

\subsection{ The natural map  $\eta : \Der_k(B,N) \to
     \Der_k(R, \Cent_k(B,N))$} \label{subsec:DDD}
In this section $B$ is an $R$--algebra for some $R\in \kalg$ and $N$
is a $(B,R)$--dimodule.

\begin{prop} \label{eta-def}
{\rm (a)} For $d\in \Der_k(B,N)$ and $r\in R$ the map $\eta_{B,N}(d)(r) \co B
\to N$, defined by $\big (\eta_{B,N}(d)(r)\big)(b) \allowbreak = d(rb) - r
d(b)$, lies in $\Cent_k(B,N)$. \sm

{\rm (b)} The map
\begin{equation} \label{eta-def0}
  \eta_{B,N} \co \Der_k(B,N) \to \Der_k\big(R, \Cent_k(B,N)\big), \quad
      d \mapsto \eta_{B,N}(d)
\end{equation} is a well-defined $k$-linear map. It gives rise to the exact sequence
\begin{equation} \label{eta-def1}
 0 \to \Der_R(B,N) \to \Der_k(B,N) \xrightarrow{\eta_{B,N}}
     \Der_k\big( R, \Cent_k(B,N)\big).
\end{equation}
\end{prop}

\begin{proof} We denote $\eta_{B,N}$ by $\eta$. (a) Let $r\in R$ and put
$\bar d = \eta(d)(r)$. Then for all $b_1, b_2\in B$ we have
\begin{align*}
 \bar d(b_1 b_2) &= d(rb_1 b_2) - r d(b_1 b_2)
             = d(b_1 (rb_2)) - r d(b_1 b_2) \\
   &= b_1 \cdot d(rb_2) + d(b_1) \cdot (rb_2) - r (d(b_1) \cdot b_2)
          - r\big(b_1 \cdot d(b_2)\big) \\
  & = b_1 \cdot d(rb_2) - b_1 \cdot \big( rd(b_2)\big)
       = b_1 \cdot \big(\bar d(b_2) \big)
         \end{align*}
where we used $d(b_1) \cdot (rb_2) = r (d(b_1) \cdot b_2)$ since $N$ is a
$(B,R)$--dimodule. One can similarly show $\bar d(b_1 b_2) = \big( \bar
d(b_1) \big) \cdot b_2$, thus proving that $\bar d \in \cent_k(B,N)$.

In (b) we first verify that $\eta(d)=: \tilde d$ is a derivation, which means
$\tilde d (r_1r_2) = \tilde d (r_1) \cdot  r_2 + r_1 \cdot \tilde d(r_2)$ in
$\Cent_k(B,N)$. For $b\in B$ we get in view of Lemma~\ref{der-R-mod}
\begin{align*}
 \big( \tilde d(r_1 r_2)\big)(b) &= d(r_1 r_2 b) - r_1 r_2 d(b)
   = d(r_1 r_2 b) - r_1 d(r_2 b)  + r_1 d(r_2 b) - r_1 r_2 d(b)\\
  & =\big(\tilde d(r_1) \cdot r_2 \big)(b) + \big( r_1 \cdot \tilde d(r_2)\big)(b),
\end{align*}
which proves our claim. We therefore have a well-defined $k$--linear map
\[ \eta : \Der_k(B,N) \to \Der_k\big( R, \Cent_k(B,N)\big).\]
 That $\eta$ also
has the other property stated in (b) is then immediate.
  \end{proof}

\subsection{A section of $\eta$: Untwisted case}
The map $\eta$ of Proposition~\ref{eta-def} admits a natural section whenever
the algebra $B$ comes from $k$. More precisely:\footnote{The change of
notation from section~\ref{subsec:DDD} replacing $B$ by $A$, $R$ by $k$ and $N$ by $M$ is put into place to match future references.}

\begin{lemma} \label{eta-split} Let $A$ be a perfect $k$--algebra, $S\in \kalg$
and $M$ an $(A_S, S)$--dimodule. Then
\[ \si_{A_S, M} \co \Der_k \big(S, \Cent_k (A_S, M)\big) \to \Der_k (A_S, M), \quad
    \si_{A_S, M}(d)(a \ot s)= d(s)(a\ot 1_S) \]
is a well-defined $k$--linear map and a section of the map $\eta_{A_S, M}$ of
\eqref{eta-def0}, in particular $\eta_{A_S, M}$ is surjective and
\begin{equation} \label{eta-split1}
    \Der_k (A_S, M) \simeq \Der_S (A_S, M) \oplus
                \Der_k \big(S, \Cent_k (A_S, M)\big)
\end{equation}
as $S$-modules.
\end{lemma}

\begin{proof} The map $A \times S \to M$, $(a,s) \mapsto d(s)(a \ot 1_S)$ is
well-defined and $k$-balanced, hence gives rise to a well-defined $k$--linear
map $\si_{A_S, M}(d)$ which we denote by $\si(d)$. Next we check that $\si(d)$ is a derivation.
With the obvious notation we have
 \begin{align*}
    \si(d)\big( (a_1 \ot s_1) (a_2 \ot s_2)\big)
      &= \si(d)(a_1 a_2 \ot s_1 s_2) = d(s_1 s_2)(a_1 a_2 \ot 1_S)\\
  &= \big( d(s_1) \cdot s_2 + s_1 \cdot d(s_2) \big) (a_1 a_2 \ot 1_S)
  \\ &= d(s_1) (a_1 a_2 \ot s_2) + s_1 \cdot d(s_2)(a_1a_2 \ot 1_S)
  \\  &= d(s_1) (a_1 a_2 \ot s_2) + d(s_2)(a_1a_2 \ot s_1) \\
  &\qquad \hbox{(since $s_1 \cdot d(s_2)= d(s_2) \cdot s_1$
         by Lemma~\ref{cent-perf})}  \\
 & = d(s_1) \big(  (a_1 \ot 1_S) (a_2 \ot s_2) \big)
      + d(s_2) \big( (a_1 \ot s_1) (a_2 \ot 1_S)\big)
\\ &= \big( d(s_1)(a_1 \ot 1_S)\big)\cdot (a_2 \ot s_2)
      + (a_1 \ot s_1) \cdot \big( d(s_2) (a_2 \ot 1_S)\big)
 \\ &= \big( \si(d)(a_1 \ot s_1)\big) \cdot (a_2 \ot s_2)
        + (a_1 \ot s_1) \cdot \big( \si(d) (a_2 \ot s_2) \big).
  \end{align*}
Let $\eta= \eta_{A_S, M}$. Next we verify that $\si$ is a section of $\eta$,
i.e., $\big(( \eta \circ \si) (d)\big)(s_1) = d(s_1)$ when evaluated on $a\ot
s_2 \in A\ot_k S$:
\begin{align*}
  \big((  (\eta \circ \si)(d))(s_1) \big)(a \ot s_2) &=
     \si(d)\big (  s_1  (a \ot s_2)\big)
              - s_1  \big(\si(d)(a\ot s_2) \big)
 \\ &\qquad \hbox{(by the definition of $\eta$)}
 \\ &= \si(d)(a \ot s_1s_2) - s_1  \big(d(s_2)(a \ot 1_S)\big)
 \\ &= d(s_1s_2) (a\ot 1_S) - (s_1 \cdot d(s_2))(a \ot 1_S)
 \\ & = (d(s_1) \cdot s_2)(a\ot 1_S) = d(s_1) (a \ot s_2).  
  \end{align*}
  \end{proof}

\begin{rem}\label{azam} We describe how this lemma relates to previously obtained
results in the case when $M=A_S$, considered as natural $(A_S, S)$--dimodule
as in Example~\ref{trivial}. In this case \eqref{eta-split1} becomes
\begin{equation}\label{azam1}
\Der_k(A_S) \simeq \Der_S(A_S) \oplus \Der_k\big( S, \cent_k(A_S)\big)
\end{equation}
In \cite{Azam:tensor} S. Azam considers perfect algebras over a field $k$ and
$S\in \kalg$ for which the canonical map $ \cent_k(A) \ot_k S \to \cent_k(A
\ot_k S)$ is an isomorphism of $S$--algebras. 
Assuming this, we get
\begin{equation}\label{azam2}
\Der_k(A_S) \simeq \Der_S(A_S) \oplus \Der_k\big( S, \cent_k(A)\ot_k S \big),
\end{equation}
a decomposition which coincides with \cite[Th.~2.9]{Azam:tensor}. In
particular, if $A$ is finite-dimensional, we have $\Der_k(S, \cent_k(A) \ot_k
S) \simeq \cent_k(A) \ot_k \Der_k(S)$, thus using Proposition~\ref{proppp}
\begin{equation} \label{azam3} \begin{split}
  \Der_k(A_S) \simeq \Der_S(A_S) &\rtimes \big( \Cent_k (A) \ot_k \Der_k(S) \big)
    \\  \simeq \big( \Der_k(A) \ot_k S \big)  &\rtimes \big( \Cent_k(A) \ot_k \Der_k(S) \big).
\end{split}\end{equation}
The isomorphism \eqref{azam3} had been previously established by
Benkart-Moody \cite[Th.~1.1]{bemo} and Block \cite[Th.~7.1]{Block}.
\end{rem}
\sm

\subsection{A section of $\eta$ : Twisted case} \label{subsec:twi}
Our goal is to extend \eqref{eta-split1} to the setting of \'etale forms of
$A \otimes_k R$, namely algebras $B$ over $R$ such that $B \otimes_R S$ is
isomorphic as an $S$--algebra to $(A \otimes_k R) \ot_R S \simeq A \ot_k S$
for some {\em \'etale covering\/} $S/R$, by which we mean that $S/R$ is an
\'etale extension, i.e., flat and unramified, which is also faithfully flat.\footnote{It would be more
correct that $\Spec(S)$ is a covering of $\Spec(R)$ on the \'etale site
of $\Spec(R)$.} We pause to remind the reader that up to $R$--isomorphism we
may assume that $B$ is an $R$--subalgebra of $A \ot_k S$: The algebra $B$ can
be defined in terms of cocycles, just as one does in usual Galois cohomology
(see \cite[II]{knus-ojan} and \cite{P} for details and references).

\begin{blank} \label{goal}
Recall that $k$ denotes our base ring.  We will make the following natural
assumptions, analogous in spirit to those made in \cite[Th.~4.2]{P}.
\begin{enumerate}
      \item[(i)]  $A$ is a perfect  $k$--algebra which is finitely
          presented (in particular of finite type) as a $k$--module.
   \item[(ii)] $R\in \kalg$ is a flat extension  of $k$.
   \item[(iii)] $S \in \Ralg$ is an \'etale covering  of $R$.
 \item[(iv)] $B \subset A_S = A\ot_k S$ is an $(S/R)$-form of $A_R$.

   \item[(v)] $N$ is a $(B,R)$--dimodule.
\end{enumerate}
Our goal is to describe the nature of $\Der_k(B,N)$. To do this we need the
following additional crucial assumption. As we see later in
Lemma~\ref{rcg-suff}, the assumption is fulfilled in the most important case
of a finite \'etale covering, in other words when $B$ is isotrivial.
 \begin{enumerate}

     \item[(vi)] There exists an $R$-linear map $\pi : N_S = N \ot_R S
         \to N$ which satisfies
\begin{equation} \label{rcg-assume} \pi(b \cdot m) = b \cdot \pi(m) \quad
\hbox{and} \quad \pi(m \cdot b) = \pi(m) \cdot b
 \quad \hbox{for $b\in B$ and $m\in N_S$}
\end{equation}
where the $(B, R)$--dimodule structure of $N_S$ is given by
\eqref{di-mod-act}.
  \end{enumerate}
\end{blank}

We will immediately put assumption (iii) to good use to show how the
description of $\Der_k(B,N)$ becomes a descent problem.

\begin{prop} \label{etale-extension}  Let $P$ be an $R$--module. Then there
exists a canonical injective map
\begin{equation} \label{etale-extension1a} \varepsilon_P : \Der_k(R, P) \to
   \Der_k(S, P_S).
\end{equation}
This map is such that after identifying $R$ and $P$ with subsets of $S$ and
$P_S$ respectively we have
\begin{equation} \label{etale-extension1b}
   \big( \varepsilon_P(d) \big)(r) = d(r)
 \end{equation}
for $d\in \Der_k(R,P)$ and $r\in R$.
\end{prop}

\begin{proof} Since $S/R$ is faithfully flat (being a covering), we can canonically
identify $R$ (respectively $P$) with a subset of $S$ (respectively $P_S$).
Since $S/R$ is also \'etale, the result is well-known, see \cite[Ch.0,
\S20]{EGAIV}.
  \end{proof}
\ms

By Lemma~\ref{perf-descent}, $B$ is perfect. It is also finitely presented as
explained in the proof of Corollary~\ref{forms-are-cent}.  From
Lemma~\ref{cent-perf} and Lemma~\ref{cent-ext} together with our assumptions
it follows that
\[ \Cent_k (B,N) \ot_R S = \Cent_R (B,N) \ot_R S \simeq
     \Cent_S (B_S, N_S) \simeq \Cent_k (A_S, N_S).
\]
Hence, using Proposition~\ref{etale-extension}, we have an injective
map \begin{equation}
 \label{goalIIa} \rh_N = \rh_{\Cent_k(B,N)} :
    \Der_k \big( R, \Cent_k (B,N) \big) \hookrightarrow
   \Der_k \big( S, \Cent_k (A_S, N_S)\big) \end{equation}
The situation that we are presently at, can be summarized by the following
diagram.\footnote{``What can be shown cannot be said" L. Wittgenstein.}
\begin{equation} \label{goalIIb}
 \vcenter{\xymatrix{
    0 \ar[r] & \Der_S  (A_S, N_S) \ar[r]
       & \Der_k (A_S, N_S)  \ar@{~>}[d]
    \ar[r]_>>>{\eta_{S}} & \Der_k \big( S, \Cent_k (A_S, N_S) \big)
     \ar@/_1pc/[l]_{\si_{S}} \ar[r] &0 \\
   0 \ar[r] & \Der_R (B, N) \ar[r] & \Der_k (B,N) \ar[r]_>>>>>{\eta_B}
     & \Der_k \big( R, \Cent_k (B,N)\big) \ar[u]_{\varepsilon_N}
       \ar@{.>}@/_1pc/[l]_{\si}   } }
\end{equation}
where we have abbreviated $\eta_S = \eta_{A_S, N_S}$, $\eta_B = \eta_{B,N}$
and $\si_S = \si_{A_S, N_S}$. The exactness of the rows follows from
Proposition~\ref{eta-def}, and the splitting of $\eta_S$ from
Lemma~\ref{eta-split}. Recall that by the faithfull flatness of $S/R$ there
is no harm to assume $R \subset S$ and $N \subset N_S$. \sm

Our immediate goal is to  construct the dotted arrow $\si$. We will do this
by going to $\Der_k (A_S, N_S)$ using $\si_S \circ \rh_N$ and then require
that the derivations obtained in this way map $B$ to $N$ as indicated by the arrow $\rightsquigarrow$ (recall that we know
$B \subset A_S$ and $N \subset N_S$). Thus we need the condition
\begin{equation} \label{goalIIc}
 \big( (\si_S \circ \rh_N)(d) \big) (B) \subset N
 \quad \hbox{for all $d\in  \Der_k \big( R, \Cent_k (B,N)\big)$}
\end{equation}
which we will establish in our main theorem below.

\begin{rem}\label{descentdata} In the present situation the
diagram (\ref{goalIIb}) can be thought of as descent data, while (\ref{goalIIc})
is the  condition for the descent data to be effective. \end{rem}

The $R$--linear map $\pi : N_S \to N$ satisfying \eqref{rcg-assume} leads to
two very important $k$--linear maps: \sm

{\rm (1)} The {\it restriction map}
\begin{equation} \label{rcg1}
  \rho : \Hom_k(A_S, N_S) \to \Hom_k(B, N), \quad
       \rho(f) : b  \mapsto (\pi \circ f)(b).
\end{equation}

{\rm (2)} The {\it double restriction  map} (shown to be well-defined in the
next theorem)
\begin{equation} \label{rcg1a}
\tilde{\rho} : \Der_k \big( S, \cent_k(A_S, N_S) \big) \to \Der_k \big(R, \cent_k(B,
N)\big),\;  \tilde{\rho}(d) : r \mapsto  \rho(d(r)). \end{equation}

\begin{theorem}\label{main} We assume {\rm (i)--(iv)} of\/ {\rm \ref{goal}}. Then
\sm

{\rm (a)} The restriction map {\rm \eqref{rcg1}} preserves derivations and centroidal
transformations:
\[
   \rho\big(\Der_k (A_S, N_S) \big) \subset
       \Der_k(B, N) \quad \hbox{and}
       \quad  \rho\big(\cent_k (A_S, N_S) \big) \subset
       \cent_k(B, N).
\]

{\rm (b)} The restriction of\/ $\rho$ to $\cent_k(A_S, N_S)$ is an
$R$--bimodule homomorphism with respect to the $R$--bimodule structures of
$\cent_k(A_S, N_S)$ and $\cent_k(B, N)$ of Lemma~{\rm
\ref{der-R-mod}:}\footnote{Since  $B$ is perfect, Lemma~\ref{cent-perf} shows
that the two $R$--module structures coincide.}
\begin{equation}\label{rcg0}
   \rho(r \cdot \chi) = r \cdot \rho(\chi) \quad \hbox{and} \quad
    \rho(\chi \cdot r) = \rho(\chi) \cdot r
 \end{equation}
for $r\in R$ and $\chi \in \cent_k(A_S, N_S)$.\sm

{\rm (c)} The double restriction map {\rm \eqref{rcg1a}} is well-defined and satisfies
\begin{equation} \label{rc4a} \eta_B \circ \rho \circ \si_S =
\tilde{\rho}\end{equation} for $\eta_B$ and $\si_S$ as in {\rm \eqref{goalIIb}.} \sm

{\rm (d)} With $\eps_N$ as in {\rm \eqref{goalIIa}} the map $\si = \rho \circ \si_S \circ \varepsilon_N$ is a section
of $\eta_B$, whence \begin{equation} \label{rcg6}
 \Der_k(B,N) \simeq \Der_R(B,N)
\oplus \Der_k\big( R, \Cent_k(B,N)\big). \end{equation}

{\rm (e)} Summarizing, we have the following diagram.
\begin{equation} \label{rcg5}
 \vcenter{\xymatrix@C=1.5pc{
    0 \ar[r] & \Der_S (A_S, N_S) \ar[r]
       & \Der_k(A_S, N_S)  \ar[r]_>>>{\eta_{S}}\ar[d]_\rho
         & \Der_k\big( S, \Cent_k(A_S, N_S) \big)
     \ar@/_1pc/[l]_{\si_{S}}  \ar@<1ex>[d]^{\tilde{\rho}}
     \ar[r] &0 \\
 0 \ar[r] & \Der_R(B, N) \ar[r] & \Der_k(B,N)
     \ar[r]_>>>>>{\eta_B} & \Der_k\big(R, \Cent_k(B,N) \big)
        \ar@<1ex>[u]^{\varepsilon_N}  \ar@/_1pc/[l]_{\si}   } }
\end{equation}
\end{theorem}

\begin{proof} (a) Let $d\in \Der_k(A_S, N_S)$ and put $\tilde d = \rho(d)$.
For $b_i \in B$ we then get, using \eqref{rcg-assume},
\begin{align*}
   \tilde d(b_1 b_2) &= \pi\big( d(b_1 b_2)\big)
      = \pi\big( d(b_1) \cdot  b_2\big)  + \pi\big( b_1 \cdot d(b_2) \big)
     \\ &= \pi\big( d(b_1) \big) \cdot b_2 + b_1 \cdot \pi\big( d(b_2) \big)
     = \tilde d (b_1) \cdot b_2 + b_1 \cdot \tilde d(b_2).
 \end{align*}
This shows  $\rho\big(\Der_k (A_S, N_S) \big) \subset \Der_k(B, N)$. The
second inclusion can be proven in a similar fashion: Let $\chi
\in \cent_k(A_S, N_S)$ and put $\tilde \chi = \rho(\chi)$. Then
 \[
    \tilde \chi (b_1 b_2) = \pi\big( \chi (b_1  b_2)\big)
    = \pi\big( \chi(b_1) \cdot b_2\big) = \pi\big(\chi(b_1)\big) \cdot b_2
    = \tilde \chi(b_1) \cdot b_2.\]
The equation $\tilde \chi (b_1 b_2) = b_1 \cdot \tilde \chi(b_2)$ follows in
the same way. \sm

(b) For the proof of \eqref{rcg0} we use that $\pi$ is $R$--linear. We have
for $\chi \in \cent_k(A_S, N_S)$ and $b\in B$
\begin{align*}
  \big( \rho(r \cdot \chi) \big)(b)
   &= \pi \big( (r \cdot \chi) (b) \big)
   = \pi \big(  r \cdot (\chi(b))\big)
   = r \cdot \pi\big( \chi(b)\big)
   = r \cdot \big(\rho(\chi)(b)\big)
  \\ & = \big(r \cdot \rho(\chi)\big)(b), \\
  \big( \rho(\chi \cdot r) \big)(b)
   &= \pi\big(  (\chi \cdot r)(b)\big) = \pi\big( \chi(rb)\big)
  = \rho(\chi)(rb) = \big( \rho(\chi) \cdot r\big)(b).
\end{align*}

(c) We need to show that $\tilde{\rho}$ is well-defined, i.e., that
$\tilde{\rho}(d)$ is a derivation (it is clearly a $k$--linear map $R \to
\cent_k(B, N)$). Thus, for $\bar d=\tilde{\rho}(d)$ and $r_i \in R$ we need
to prove that $\bar d(r_1 r_2) = \bar d(r_1) \cdot r_2 + r_1 \cdot \bar
d(r_2)$. To this end we use \eqref{rcg0}:
\begin{align*}
 \bar d(r_1 r_2) &= (\rho \circ d)(r_1 r_2) = \rho\big( d(r_1r_2) \big)
     = \rho\big (    d(r_1) \cdot r_2 + r_1 \cdot d(r_2) \big)
   \\ &= \big( \rho(d(r_1))\big) \cdot r_2 + r_1 \cdot \big(\rho( d(r_2)) \big)
     = \bar d( r_1)  \cdot r_2 + r_1 \cdot \bar d(r_2).
\end{align*}
Finally, for the proof of \eqref{rc4a} let $d\in \Der_k\big(S, \cent_k(A_S,
N_S)\big)$, $r\in R$, $b\in B$ and put $d'=\si_S(d)\in \Der_k (A_S, N_S)$.
Then, using that $\si_S$ is a section of $\eta_S$, we get
\begin{align*}
  &\Big(\big( \eta_B (\rho(d')) \big)(r)\Big)(b)
  = \rho(d')(rb) - r \big( \rho(d')(b)\big)
   =  \pi\big( d'(rb)\big) - r \pi\big( d'(b)\big)
   \\& \quad = \pi \big( d'(rb) - r d'(b)\big) =
       \pi\big( \eta_S(d')(r) (b)\big)
   = \pi\big( d(r)(b)\big) = \big(\big(\tilde{\rho}(d)\big)(r)\big)(b).
 \end{align*}

(d) follows from $\eta_B \circ \si = \tilde{\rho} \circ \varepsilon_N = \Id$
in view of \eqref{etale-extension1b} and \eqref{rc4a}. (e) follows from the diagram \eqref{goalIIb} and (d).
  \end{proof}

We finish this section by discussing important situations in which the
assumptions of the theorem hold. Assumptions (i) through (v) are quite mild
and natural within the theory of forms. The key to effective descent is (vi).
As we shall presently see, it holds in one of the most important cases, namely 
in the case of a finite \'etale covering $S/R$. In particular, it
holds when the trivializing extension $S/R$ is Galois or when $S/R$ is a
finite separable extension of fields.

\begin{lem}\label{rcg-suff}
{\rm (a)}  Assume $B$ is an $R$--algebra for $R\in \kalg$, $N$ is a
$(B,R)$--dimodule and $S\in \Ralg$ is a faithfully flat extension. After
canonically identifying $R$ with a subring of $S$ we further suppose that $R$
is a direct summand of the $R$-module $S$, so that we have
\begin{equation} \label{rcg-suff1a} S=R \oplus S',
 \end{equation}
as a direct sum of $R$-submodules. Define $\pi : N_S \to N$ as the projection
onto $N$ with respect to the decomposition $N_S = N \oplus (N \ot_R S')$.
Then $\pi$ is $R$--linear and satisfies \eqref{rcg-assume}. \sm

{\rm (b)} Suppose $S/R$ is a faithful, finitely generated and projective
$R$-module, e.g., a finite \'etale covering, then \eqref{rcg-suff1a} holds.
\end{lem}

\begin{proof} (a) Since $N_S = N \oplus (N \ot_R S')$ is a decomposition
of $N_S$ as $R$-module, the map $\pi$ is $R$--linear. From the $(B,
R)$--dimodule structure of $N_S$ given in \eqref{di-mod-act} it follows that
$N=N\ot_R R$ and $N\ot_R S'$ are subdimodules, which easily implies
that \eqref{rcg-assume} holds. \sm

(b) This is \cite[III, Lemme~1.9]{knus-ojan}.
  \end{proof}

\begin{rem} We point out that the assumption in Lemma~\ref{rcg-suff}(b)
is not necessary for \eqref{rcg-suff1a} to hold. For example, let $R=k[t]$,
$S=k[t] \times k[t^{\pm 1}]$ viewed as $R$--algebra by embedding $R$
diagonally into $S$. Then $S$ is an \'etale covering of $R$ which is not
finite. But \eqref{rcg-suff1a} holds, for
example by taking $S'= 0 \oplus k[t^{\pm 1}]$.
\end{rem}

\begin{example} The decomposition \eqref{rcg-suff1a} takes place naturally
whenever a finite group $\Ga$ acts completely
reducibly on $S\in \kalg$ by $k$--algebra automorphisms. One then knows that
$S = S^\Ga \oplus S'$ where
\begin{align*}
  S^\Ga &= \{ s\in S : \ga \cdot s = s \hbox{  for all } \ga\in
\Ga\}, \\
   S' &= \Span_k \{ \ga \cdot s - s: \ga \in \Ga, s\in S\}.
 \end{align*}
It is immediate that $R=S^\Ga$ is a $k$--algebra and that $S'$ is an
$R$-submodule: $r(\ga \cdot s - s) = \ga \cdot (rs) - rs \in S'$ where $r\in R$ and $s\in S$. Such a situation occurs for example in equivariant map algebras \cite{NSS}.
\end{example}

\begin{rem} Let again $\Ga$ be a finite group of automorphisms of
$S\in \kalg$ and let $R=S^\Ga$. Then \cite[V.2]{SGA1} gives sufficient
conditions for $S/R$ to be a finite \'etale covering. The most relevant case
is that of a finite Galois extension $S/R$ with Galois group $\Ga$. The quintessential example is that of a multiloop algebra \cite{P}. \end{rem}

\section{Applications}\label{applications}

In this section we discuss several special cases of our main result as well
as generalizations and applications.

\subsection{Algebra derivations}
We assume the conditions (i)--(vi) of \ref{goal} where $N=B$ with its natural
$(B,k)$--dimodule structure of Example~\ref{trivial}. Moreover, we suppose
that $A$ is central. By definition, $\cent_k(B,B) = \cent_k(B)$ is the usual
centroid of $B$. From Corollary~\ref{forms-are-cent} and
Lemma~\ref{cent-perf} we get
\[ R \simeq \Cent_R(B) \simeq \Cent_k(B).\]
The $k$-module $\Der_k(B)$ has a natural Lie $k$--algebra structure.
Our Theorem~\ref{main} states
\begin{equation} \label{algebra1}
    \Der_k(B) \simeq \Der_R(B) \rtimes \Der_k(R).
\end{equation}
If $S/R$ is a Galois extension, \ref{goal}(vi) holds and \eqref{algebra1}
generalizes the main theorem of \cite{P} (in the Galois case).

Of course, the most important case is that when $k$ is a field and $A$ is a
finite-dimensional (perfect and central) $k$-algebra. For example, if $B$ is
a multiloop algebra, say $B=M(A, \si_1, \ldots, \si_n)$ for commuting finite order automorphisms $\si_i$ of the $k$-algebra $A$, it can be shown that
$\Der_R(B) \simeq M(\Der_k(A), \si_1^*, \ldots, \si_n^*)$ where $\si_i^*(d) =
\si_i \circ d \circ \si_i^{-1}$. We thus recover
\cite[Th.~3.7]{Azam:multiloop} in the case of a central $A$. We point out
that in case $A$ is not necessarily central, we can also get the description
of $\Der_k(B)$ of \cite[(3.8)]{Azam:multiloop} by interpreting the second sum
in loc. cit. as $\Der_k\big(R, \cent_k(B)\big)$.
\sm

\subsection{Lie algebras} \label{subsec:lie}
We specialize Theorem~\ref{main} to the setting of forms of Lie algebras. We
will change the notation to abide by standard conventions. Recall that a
module $M$ of a Lie algebra $L$ over a field  is called {\em locally
finite\/} if every $m\in M$ belongs to  a finite-dimensional $L$-submodule of
$M$.

\begin{lemma} \label{lem:lf-g-mod} Let $\g$ be a finite-dimensional semisimple Lie algebra
over a field $k$ of characteristic zero, let $K\in \kalg$ and let $M$ be a
$\g_K$-module which is locally-finite as a $\g$-module. Then $\Der_K(\g_K, M)
= \IDer(\g_K, M)$. \end{lemma}

\begin{proof} Let $d \in \Der_K(\g_K, M)$. Then $d(\g\ot 1_K)$ is a
finite-dimensional subspace of $M$. Hence there exists a finite-dimensional
$\g$-submodule $N$ of the $\g$-module $M$ such that $d(\g \ot 1_K) \subset
N$. By the first Whitehead Lemma, there exists $n\in N$ such that $d(x \ot
1_K) = (x\ot 1_K) \cdot n$ for all $x\in \g$. But then for $s\in K$ we get
$d(x\ot s) = s d(x\ot 1_K) = s \big( (x\ot 1_K) \cdot n\big) = (x\ot s) \cdot
n$, which shows that $d$ is the inner derivation given by $n$.   \end{proof}

\begin{prop} \label{cor-Lie} Let $\g$ be a finite-dimensional semisimple Lie algebra
defined over a field\/ $k$ of characteristic zero,
let $R\in \kalg$, $S\in \Ralg$ an \'etale covering, $\scL$ an $(S/R)$-form of $\g_R
= \g\ot_k R$ and $N$ an $\scL$-module such that the canonical $\g_S$-action
on $N_S$ is locally-finite. Also suppose that we have an $R$--linear map $\pi
: N_S \to N$ satisfying \eqref{rcg-assume}. Then
\[
 \Der_k(\scL, N) \simeq \IDer(\scL, N) \oplus
         \Der_k\big(R, \cent_k(\scL, N)\big).\]
In particular the first cohomology group of $\scL$ with coefficients in $N$
is
\[
     H^1(\scL, N) \simeq \Der_k\big( R, \cent_k(\scL, N)\big).
\]
\end{prop}

\begin{proof} All assumptions of Theorem~\ref{main} are fulfilled. The result
then follows from \eqref{rcg6} as soon as we have shown that $\Der_R(\scL, N)
= \IDer(\scL, N)$. Since $\scL$ is finitely presented, an application of
Corollary~\ref{cor-innder-Lie} shows that this holds if and only if
$\Der_S(\scL_S, N_S) = \IDer(\scL_S, N_S)$, equivalently, $\Der_S(\g_S, N_S)=
\IDer(\g_S, N_S)$. But the latter equality is a consequence of
Lemma~\ref{lem:lf-g-mod}.
  \end{proof}
\sm

\subsection{Associative algebras}\label{asso-rsult}
Theorem~\ref{main} applies also to associative algebras. For them, it is
natural to assume that $N$ is a $B$-bimodule. Suppose that this is the case
and that $A$ is a separable $k$--algebra. It then follows that $B$ is
separable too \cite[III, \S2]{knus-ojan}. Hence (\cite[III,
Thm.~1.4]{knus-ojan}) all  $R$--linear derivations $d\co B \to N$ are inner,
i.e., $\Der_R(B,N) = \IDer(B,N)$, and so the first Hochschild cohomology
group of the associative $k$--algebra $B$ with values in the $k$-module $N$
is
\[
     HH^1(B, N)\simeq \Der_k\big(R, \Cent_k(B,N)\big).
\]
For example. for $B=N$ we have $\cent_k(B,N) = \cent_k(B) = Z(B)$, the centre
of $B$. We therefore get
\[ HH^1(B) \simeq \Der_k\big(R,Z(B)\big). \]
In particular, for a central $B$ we get $HH^1(B) \simeq \Der_k(R)$.
\sm

\subsection{Jordan algebras}
We will leave the interpretation of Theorem~\ref{main} for general Jordan
algebras to the reader, and only consider the most interesting special case
here.

The analogue of Lemma~\ref{lem:lf-g-mod} for Jordan algebras remains true by
replacing Whitehead's Lemma by Jacobson's Theorem \cite[VIII, Th.~10]{jake},
which says that every derivation of a semisimple finite-dimensional Jordan
algebra with values in a Jordan bimodule is inner. We therefore also have the
analogue of Proposition~\ref{cor-Lie} which we state here in simplified form.

\begin{prop} \label{jordan-p} Let $J$ be a finite-dimensional
semisimple Jordan algebra over a field $k$ of characteristic $0$,
$R\in \kalg$ an extension of $k$, 
$S\in \Ralg$ a finite \'etale covering, $\scJ$ an $S/R$-form of
$J_R$ and $N$ a Jordan bimodule of $\scJ$ such that the canonical action of
$J_S$ on $N_S$ is locally finite. Then
\[ \Der_k(\scJ, N) \simeq \IDer(\scJ, N) \oplus \Der_k\big(R, \cent_k(\scJ,
N)\big).
\] \end{prop}

We have assumed characteristic zero only to simplify the presentation. This
can be generalized depending on the type of $J$. For example, if $J$ is
separable simple exceptional then characteristic $\ne 2$ suffices (one must
then however replace Jacobson's Theorem by Harris' Theorem \cite[VII.7,
Th.~14]{jake}).
\sm

\subsection{Work in preparation}
Our main theorem is also at the heart of several works in preparation. We
outline three of them.

 (a) In this case $A=\g$ is a finite-dimensional central-simple
      Lie algebra over a field $k$ of characteristic $0$. We take
      $N=B^*$. If $V$ is a $k$-space, viewed as a trivial $B$-module, the
      extensions
\[ 0 \to V \to E \to B \to 0 \]
are measured (up to isomorphisms) by $H^2(B,V)$. Our main theorem can be
used to show that every cocycle in $Z^2(B,V)$ is cohomologous to a {\em
unique} standard cocycle. Examples of standard cocycles are given by the
universal cocycle of Kassel \cite{Kas}. Details will be given in
\cite{pps-new}. \sm

(b) In \cite{NPP} we will use the relation between invariant
    bilinear forms of an algebra $A$ and $\cent_k(A, A^*)$ to describe
    invariant bilinear forms on algebras obtained by \'etale descent.
    Among other things we will recover \cite[Lemma~2.3]{MSZ} which
    considers this question for Lie algebras in the untwisted case. The
    need to consider graded invariant forms will require a graded version
    of our main theorem. Besides other applications we will give
    a classification-free proof of Yoshii's Theorem \cite{y:lie} for
    multiloop Lie tori stating that graded invariant
     bilinear forms are unique up to scalars. \sm

(c) The results of this paper are also relevant for the
    description of the universal central extension $\uce(L)$ of a perfect
    graded Lie algebra $L$. Namely, denote by $\AD(L)= \{d\in
    \Der_k^\La(L, L\gr*): d(l)(l)=0 \}$,  the alternating $\La$-graded derivations of
    $L$ into its graded dual. Then (\cite{N1}, see also
    \cite[5.1.3]{n:eala-summer}) the homology $H_2(L)$, which is the
    kernel of the universal central extension, is canonically isomorphic
    to $\AD(L)/\IDer(L, L\gr*)$. Moreover, this approach naturally leads
    to a description of the universal central extension $\uce(L^\scF)$ of
    a fixed point subalgebra of $L$ under a group of automorphisms $\scF$
    as the fixed point subalgebra $\uce(L)^\scF$, where $\scF$ is
    canonically extended to $\uce(L)$. In particular, this applies to
    certain equivariant map algebras, like multiloop algebras, and to
    algebras obtained by Galois descent.

\end{document}